\documentclass[11pt,reqno]{amsart}
\usepackage{amsthm,amssymb,url,hyperref}
\usepackage{fullpage}

\theoremstyle{plain}
\newtheorem{theorem}{Theorem}[section]

\newtheorem{corollary}[theorem]{Corollary}

\newtheorem{lemma}[theorem]{Lemma}

\newtheorem{proposition}[theorem]{Proposition}

\theoremstyle{definition}

\newtheorem{remark}[theorem]{Remark}

\newtheorem{example}[theorem]{Example}

\numberwithin{equation}{section}

\def\ldiv{\backslash}

\def\ld#1{\backslash^{\!\!#1}}
\def\rd#1{/_{\!\!#1}}

\def\aut#1{{\mathrm{Aut}(#1)}}

\def\lmlt#1{\mathrm{LMlt}(#1)}

\def\dis#1{\mathrm{Dis}(#1)}
\def\disp#1{\mathrm{Dis^+}(#1)}
\def\disn#1{\mathrm{Dis^-}(#1)}

\title{Idempotent solutions of the Yang-Baxter equation and twisted group division}

\author{David Stanovsk\'y}

\address{Department of Algebra, Faculty of Mathematics and Physics, Charles University, Prague, Czechia}
\email{stanovsk@karlin.mff.cuni.cz}

\author{Petr Vojt\v echovsk\'y}
\address{Department of Mathematics \\ University of Denver \\ Denver (CO), USA}

\email{petr@math.du.edu}

\thanks{The paper is written within the framework of the cooperation grant LTAUSA19070. David Stanovsk\'y partially supported by the GA\v CR grant 18-20123S. Petr Vojt\v echovsk\'y partially supported by the 2019 PROF grant of the University of Denver.}


\date{\today}

\begin{document}

\begin{abstract}
Idempotent left nondegenerate solutions of the Yang-Baxter equation are in one-to-one correspondence with twisted Ward left quasigroups, which are left quasigroups satisfying the identity $(x*y)*(x*z)=(y*y)*(y*z)$. Using combinatorial properties of the Cayley kernel and the squaring mapping, we prove that a twisted Ward left quasigroup of prime order is either permutational or a quasigroup. Up to isomorphism, all twisted Ward quasigroups $(X,*)$ are obtained by twisting the left division operation in groups (that is, they are of the form $x*y=\psi(x^{-1}y)$ for a group $(X,\cdot)$ and its automorphism $\psi$), and they correspond to idempotent latin solutions. We solve the isomorphism problem for idempotent latin solutions.
\end{abstract}

\maketitle

\section{Introduction}

We continue our program of studying left nondegenerate set-theoretic solutions of the Yang-Baxter equation from an algebraic perspective, taking advantage of the associated left quasigroups \cite{BKSV,HSV}.

It is well-known that the algebraic counterpart to derived left nondegenerate solutions of the Yang-Baxter equation are \emph{racks}, which are left quasigroups satisfying the identity
\begin{equation}\label{Eq:RACK}
    (x*y)*(x*z)=x*(y*z).
\end{equation}
There is vast literature on racks, on idempotent racks (i.e., \emph{quandles}), and on their usage in knot theory \cite{AG,EN,FR,Nosaka,Sta-latin}.

Jones \cite{Jones} and Turaev \cite{Turaev} showed that Yang-Baxter operators $r:V\otimes V\to V\otimes V$ that satisfy a quadratic equation $r^2=ar+b$ give rise to polynomial invariants of knots. In the realm of set-theoretic solutions, the quadratic equation reduces to either $r^2=1$ (the \emph{involutive} case) or $r^2=r$ (the \emph{idempotent} case).

Rump showed in \cite{Rump} that the algebraic counterpart to involutive left nondegenerate solutions are \emph{cycle sets}, which are left quasigroups satisfying the identity
\begin{equation}\label{Eq:RUMP}
    (x*y)*(x*z) = (y*x)*(y*z).
\end{equation}
We have renamed cycle sets \emph{Rump left quasigroups} in \cite{BKSV} and we continue to use this terminology here.

In this paper we focus on idempotent left nondegenerate solutions. In Section \ref{sec:uniform} we provide a uniform treatment to the correspondences for derived, involutive and idempotent left nondegenerate solutions. We recover the above correspondences for derived and involutive left nondegenerate solutions, and we prove that the algebraic counterpart to idempotent left nondegenerate solutions are left quasigroups satisfying the identity
\begin{equation}\label{Eq:tW}
    (x*y)*(x*z) = (y*y)*(y*z),\tag{tW}
\end{equation}
which we call \emph{twisted Ward left quasigroups} for reasons explained below.

In Section \ref{sec:tw-quasigroups} we completely classify twisted Ward quasigroups and hence idempotent latin solutions of the Yang-Baxter equation, cf. Theorem \ref{Th:main}. First we show that, unlike in the case of latin racks and Rump quasigroups, every twisted Ward quasigroup is isotopic to a group. In fact, every twisted Ward quasigroup has the form $(X,*)$ with $x*y = c\psi(x^{-1}y)$, where $(X,\cdot)$ is a group, $\psi$ is an automorphism of $(X,\cdot)$ and $c\in X$. Classifying these quasigroups up to isomorphism, it suffices to consider groups $(X,\cdot)$ up to isomorphism, automorphisms $\psi$ up to conjugation in the automorphism group of $(X,\cdot)$, and $c=1$. Summarizing, \emph{idempotent latin solutions of the Yang-Baxter equation are in one-to-one correspondence with conjugacy classes of automorphisms of groups}.

The above representation theorem is the reason why we have chosen the terminology ``twisted Ward quasigroup.'' Quasigroups $(X,*)$ defined over groups $(X,\cdot)$ via $x*y = x^{-1}y$ were first investigated by Ward in \cite{Ward} and later became known as \emph{Ward quasigroups}. They are precisely the quasigroups satisfying the identity
\begin{displaymath}
    (x*y)*(x*z)=y*z
\end{displaymath}
and they can be though of as groups axiomatized by left division instead of multiplication, cf. \cite{JV,Polonijo,Rabinow,Ward}.

In Section \ref{sec:prime} we classify twisted Ward left quasigroups of prime order. We consider two equivalence relations on a twisted Ward left quasigroup, namely the Cayley kernel and the kernel of the squaring map, and we show that, in the finite case, each of the equivalences has blocks of uniform size. Consequently, a twisted Ward left quasigroup of prime order is either permutational or a quasigroup. This complements the result of Etingof, Soloviev and Guralnick \cite{ESG} who classified all indecomposable nondegenerate solutions to the Yang-Baxter equation with a prime number of elements.

\section{Three classes of left nondegenerate braidings}\label{sec:uniform}

A binary algebraic structure $(X,*)$ is a \emph{left quasigroup} if all left translations $L_x:X\to X$, $y\mapsto x*y$ are bijections of $X$. In a left quasigroup $(X,*)$ we can define the left division operation by $x\ld*y = L_x^{-1}(y)$ and observe that, obviously,
\begin{equation}\label{Eq:LeftQuasigroup}
    x\ld*(x*y) = y = x*(x\ld*y)
\end{equation}
holds for all $x,y\in X$. Conversely, if $(X,*,\ld*)$ is a set equipped with two binary operations satisfying the identity \eqref{Eq:LeftQuasigroup}, then $(X,*)$ is a left quasigroup with left division $\ld*$. Note that if $(X,*)$ is a left quasigroup with left division $\ld*$, then $(X,\ld*)$ is a left quasigroup with left division $*$ \cite{Pfl}. 

Dually, $(X,*)$ is a \emph{right quasigroup} if all right translations $R_x:X\to X$, $y\mapsto y*x$ are bijections of $X$. The right division operation is then defined by $x\rd* y = R_y^{-1}(x)$. A left quasigroup that is also a right quasigroup is called a \emph{quasigroup}.

A \emph{set-theoretic solution} of the Yang-Baxter equation
\begin{equation}\label{Eq:YBE}
    (r\times 1)(1\times r)(r\times 1) = (1\times r)(r\times 1)(1\times r)\tag{YB}
\end{equation}
is a mapping $r:X\times X\to X\times X$ such that \eqref{Eq:YBE} holds as an equality of mappings $X\times X\times X\to X\times X\times X$ under composition \cite{Drinfeld}. Set-theoretic solutions of \eqref{Eq:YBE} are also known as \emph{braidings} \cite{L}.

For any mapping (not necessarily a braiding) $r:X\times X\to X\times X$ we write
\begin{displaymath}
    r(x,y) = (x\circ y,\,x\bullet y)
\end{displaymath}
for suitable binary operations $\circ$ and $\bullet$ on $X$. Straightforward calculation then shows that $r$ is a braiding if and only if the following three identities hold:
\begin{align}
    x\circ(y\circ z) &= (x\circ y)\circ((x\bullet y)\circ z),   \tag{YB1}\label{Eq:YB1} \\
    (x\circ y)\bullet((x\bullet y)\circ z) &= (x\bullet (y\circ z))\circ(y\bullet z), \tag{YB2}\label{Eq:YB2} \\
    (x\bullet y)\bullet z &= (x\bullet(y\circ z))\bullet(y\bullet z). \tag{YB3}\label{Eq:YB3}
\end{align}

A mapping $r:X\times X\to X\times X$ is said to be \emph{left nondegenerate} if $(X,\circ)$ is a left quasigroup; \emph{nondegenerate} if $(X,\circ)$ is a left quasigroup and $(X,\bullet)$ is a right quasigroup; \emph{latin} if $(X,\circ)$ is a quasigroup; \emph{derived} if $x\bullet y=x$ for every $x,y\in X$; \emph{involutive} if $r^2=1$ \cite{Rump}; and \emph{idempotent} if $r^2=r$ \cite{L}.

We are mostly interested in braidings where one of the two operations $\circ$, $\bullet$ is either trivial or can be reconstructed from the other one. The following result gives three classes of mappings with such a property.

\begin{lemma}\label{Lm:OneOp}
Let $r:X\times X\to X\times X$ be a mapping, $r(x,y)=(x\circ y,x\bullet y)$. Then:
\begin{enumerate}
\item[(i)] $r$ is derived if and only if $x\bullet y=x$.
\item[(ii)] $r$ is involutive and left nondegenerate if and only if $(X,\circ)$ is a left quasigroup and $x\bullet y = (x\circ y)\ld{\circ} x$.
\item[(iii)] $r$ is idempotent and left nondegenerate if and only if $(X,\circ)$ is a left quasigroup and $x\bullet y = (x\circ y)\ld{\circ}(x\circ y)$.
\end{enumerate}
\end{lemma}
\begin{proof}
Part (i) holds by definition. For (ii), note that $r$ is involutive if and only if the identities
\begin{displaymath}
   (x\circ y)\circ(x\bullet y) = x,\quad\quad\quad (x\circ y)\bullet(x\bullet y) = y
\end{displaymath}
hold. If $r$ is also left nondegenerate then $x\bullet y = (x\circ y)\ld{\circ} x$ is equivalent to the first identity, and the second identity becomes
\begin{displaymath}
    ((x\circ y)\circ ((x\circ y)\ld{\circ}x))\ld{\circ}(x\circ y) = y,
\end{displaymath}
which holds in any left quasigroup $(X,\circ)$. Finally, for (iii), note that $r$ is idempotent if and only if the identities
\begin{displaymath}
    (x\circ y)\circ(x\bullet y) = x\circ y,\quad\quad\quad (x\circ y)\bullet(x\bullet y) = x\bullet y
\end{displaymath}
hold. If $r$ is also left nondegenerate then $x\bullet y = (x\circ y)\ld{\circ}(x\circ y)$ is equivalent to the first identity, and the second identity becomes
\begin{displaymath}
    (z\circ (z\ld\circ z))\ld\circ (z\circ (z\ld\circ z)) = z\ld\circ z
\end{displaymath}
with $z=x\circ y$, which holds in any left quasigroup $(X,\circ)$.
\end{proof}

The proof of the following result is somewhat involved. Nevertheless it is purely equational and can be verified in a fraction of a second by an automated theorem prover, such as \texttt{Prover9} \cite{McCune}.

\begin{proposition}\label{Pr:OneOp}
Let $r:X\times X\to X\times X$ be a mapping. Then:
\begin{enumerate}
\item[(i)] $r$ is a derived braiding if and only if $x\bullet y=x$ and
\begin{equation}\label{Eq:Rack}
    x\circ (y\circ z) = (x\circ y)\circ (x\circ z).
\end{equation}
\item[(ii)] $r$ is an involutive left nondegenerate braiding if and only if $(X,\circ)$ is a left quasigroup, $x\bullet y = (x\circ y)\ld\circ x$ and
\begin{equation}\label{Eq:Rump}
    x\circ (y\circ z) = (x\circ y)\circ (((x\circ y)\ld\circ x)\circ z).
\end{equation}
\item[(iii)] $r$ is an idempotent left nondegenerate braiding if and only if $(X,\circ)$ is a left quasigroup, $x\bullet y = (x\circ y)\ld\circ(x\circ y)$ and
\begin{equation}\label{Eq:tWard}
    x\circ (y\circ z) = (x\circ y)\circ (((x\circ y)\ld\circ(x\circ y))\circ z).
\end{equation}
\end{enumerate}
\end{proposition}
\begin{proof}
(i) By Lemma \ref{Lm:OneOp}, $r$ is a derived mapping if and only if $x\bullet y=x$ holds. Then \eqref{Eq:YB1} is equivalent to \eqref{Eq:Rack}, \eqref{Eq:YB2} is equivalent to the trivial identity $x\circ y = x\circ y$, and \eqref{Eq:YB3} is equivalent to the trivial identity $x=x$.

For the rest of the proof, let us write $xy$ instead of $x\circ y$, $x\ldiv y$ instead of $x\ld\circ y$, and $[x,y]$ instead of $x\bullet y$ to save space and improve legibility. The identities \eqref{Eq:YB1}--\eqref{Eq:YB3} then become
\begin{align*}
    x(yz) &= (xy)([x,y]z),\\
    [xy,[x,y]z] &= [x,yz][y,z],\\
    [[x,y],z] &= [[x,yz],[y,z]].
\end{align*}

(ii) By Lemma \ref{Lm:OneOp}, $r$ is an involutive left nondegenerate mapping if and only if $(X,\cdot)$ is a left quasigroup and $[x,y] = (xy)\ldiv x$. Hence \eqref{Eq:YB1} holds if and only if \eqref{Eq:Rump} holds. Suppose that \eqref{Eq:YB1} holds and let us rewrite it as $(xy)\ldiv (x(yz)) = [x,y]z$. Upon substituting $y\ldiv z$ for $z$, we obtain
\begin{equation}\label{Eq:aux}
    (xy)\ldiv (xz) = [x,y](y\ldiv z).
\end{equation}
Using \eqref{Eq:YB1}, the left hand side of \eqref{Eq:YB2} can be written as $[xy,[x,y]z] = ((xy)([x,y]z))\ldiv (xy) = (x(yz))\ldiv (xy)$. Upon substituting $y\ldiv z$ for $z$ into \eqref{Eq:YB2}, we therefore obtain the identity
\begin{displaymath}
    (xz)\ldiv (xy) = [x,z][y,y\ldiv z] = [x,z](z\ldiv y),
\end{displaymath}
which is \eqref{Eq:aux} with $y$ and $z$ interchanged. Hence \eqref{Eq:YB2} holds. To show that \eqref{Eq:YB3} also holds, first note that \eqref{Eq:Rump} can be written as $L_xL_y = L_{xy}L_{(xy)\ldiv x}$. Substituting $x\ldiv y$ for $y$, we obtain $L_xL_{x\ldiv y}=L_y L_{y\ldiv x}$, and taking inverses yields
\begin{equation}\label{Eq:RumpAux}
    L_{x\ldiv y}^{-1}L_x^{-1} =L_{y\ldiv x}^{-1}L_y^{-1}.
\end{equation}
We will return to \eqref{Eq:RumpAux} shortly. By \eqref{Eq:YB1}, the left hand side of \eqref{Eq:YB3} is equal to
\begin{displaymath}
    [[x,y],z] = ([x,y]z)\ldiv [x,y] = ((xy)\ldiv (x(yz)))\ldiv ((xy)\ldiv x).
\end{displaymath}
By \eqref{Eq:YB2} and \eqref{Eq:YB1}, the right hand side of \eqref{Eq:YB3} is equal to
\begin{align*}
    [[x,yz],[y,z]] &= ([x,yz][y,z])\ldiv [x,yz] = [xy,[x,y]z]\ldiv [x,yz]\\
    &= [xy,(xy)\ldiv(x(yz))]\ldiv [x,yz] = ((x(yz))\ldiv(xy))\ldiv ((x(yz))\ldiv x).
\end{align*}
Upon substituting $y\ldiv z$ for $z$, we see that \eqref{Eq:YB3} is then equivalent to
\begin{displaymath}
    ((xy)\ldiv (xz))\ldiv ((xy)\ldiv x) = ((xz)\ldiv (xy))\ldiv ((xz)\ldiv x),
\end{displaymath}
which is further equivalent, upon substitution of $x\ldiv y$ for $y$ and $x\ldiv z$ for $z$, to
\begin{displaymath}
    (y\ldiv z)\ldiv (y\ldiv x) = (z\ldiv y)\ldiv (z\ldiv x).
\end{displaymath}
But this says $L_{y\ldiv z}^{-1}L_y^{-1} = L_{z\ldiv y}^{-1}L_z^{-1}$, which is \eqref{Eq:RumpAux} with the variables renamed.

(iii) By Lemma \ref{Lm:OneOp}, $r$ is an idempotent left nondegenerate mapping if and only if $(X,\cdot)$ is a left quasigroup and $[x,y] = (xy)\ldiv (xy)$. Hence \eqref{Eq:YB1} holds if and only if \eqref{Eq:tWard} holds. Suppose that \eqref{Eq:YB1} holds. As in (ii), we obtain the equivalent identity \eqref{Eq:aux}. Usinq \eqref{Eq:YB1}, the left hand side of \eqref{Eq:YB2} can be written as $((xy)([x,y]z))\ldiv ((xy)([x,y]z)) = (x(yz))\ldiv (x(yz))$. Upon substituting $y\ldiv z$ for $z$ into \eqref{Eq:YB2}, we therefore obtain the identity
\begin{displaymath}
    (xz)\ldiv(xz) = [x,z][y,y\ldiv z] = [x,z](z\ldiv z),
\end{displaymath}
which is an instance of \eqref{Eq:aux} with $y=z$. Hence \eqref{Eq:YB2} holds. To see that \eqref{Eq:YB3} also holds, note that \eqref{Eq:tWard} is equivalent to $L_xL_y=L_{xy}L_{(xy)\ldiv(xy)}$, which is the same as $L_xL_{x\ldiv y} = L_yL_{y\ldiv y}$ and hence
\begin{equation}\label{Eq:tWardAux}
    L_{x\ldiv y}^{-1}L_x^{-1} = L_{y\ldiv y}^{-1}L_y^{-1}.
\end{equation}
Following the same series of steps as in (ii), we can rewrite \eqref{Eq:YB3} as
\begin{displaymath}
    ((xy)\ldiv (x(yz)))\ldiv ((xy)\ldiv (x(yz))) = ((x(yz))\ldiv (x(yz)))\ldiv ((x(yz))\ldiv (x(yz))),
\end{displaymath}
which is equivalent, upon substitution of $y\ldiv z$ for $z$, to
\begin{displaymath}
    ((xy)\ldiv (xz))\ldiv ((xy)\ldiv (xz)) = ((xz)\ldiv (xz))\ldiv ((xz)\ldiv (xz)),
\end{displaymath}
and hence to
\begin{displaymath}
    (y\ldiv z)\ldiv (y\ldiv z) = (z\ldiv z)\ldiv (z\ldiv z).
\end{displaymath}
But this is implied by $L_{y\ldiv z}^{-1}L_y^{-1} = L_{z\ldiv z}^{-1}L_z^{-1}$, which is \eqref{Eq:tWardAux} with the variables renamed.
\end{proof}

As has become clear from the proof of Proposition \ref{Pr:OneOp}, the identities \eqref{Eq:Rump} and \eqref{Eq:tWard} are somewhat inconvenient to work with. We will therefore employ the following syntactic trick due to Rump \cite{Rump} to arrive at simpler left quasigroup identities that correspond to the same braidings as in Proposition \ref{Pr:OneOp} (but not necessarily to the same left quasigroups).

Given a left nondegenerate mapping $r:X\times X\to X\times X$, the trick is to write
\begin{displaymath}
    r(x,y) = (x\ld\circ y, x\bullet y)
\end{displaymath}
and express the equations \eqref{Eq:YB1}--\eqref{Eq:YB3} in terms of the operations $\circ$ and $\bullet$, rather than $\ld\circ$ and $\bullet$.

\begin{proposition}\label{Pr:DivOneOp}
Let $r:X\times X\to X\times X$ be a left nondegenerate mapping written as $r(x,y) = (x\ld\circ y, x\bullet y)$. Then:
\begin{enumerate}
\item[(i)] $r$ is a derived left nondegenerate braiding if and only if $(X,\circ)$ is a left quasigroup, $x\bullet y=x$ and
\begin{equation}\label{Eq:DivRack}
    (x\circ y)\circ (x\circ z) = x\circ (y\circ z).
\end{equation}
\item[(ii)] $r$ is an involutive left nondegenerate braiding if and only if $(X,\circ)$ is a left quasigroup, $x\bullet y = (x\ld\circ y)\circ x$ and
\begin{equation}\label{Eq:DivRump}
    (x\circ y)\circ (x\circ z) = (y\circ x)\circ (y\circ z).
\end{equation}
\item[(iii)] $r$ is an idempotent left nondegenerate braiding if and only if $(X,\circ)$ is a left quasigroup, $x\bullet y = (x\ld\circ y)\circ(x\ld\circ y)$ and
\begin{equation}\label{Eq:DivtWard}
    (x\circ y)\circ (x\circ z) = (y\circ y)\circ (y\circ z).
\end{equation}
\end{enumerate}
\end{proposition}
\begin{proof}
(i) By Proposition \ref{Pr:OneOp}, $r$ is a derived left nondegenerate braiding if and only if $(X,\ld\circ)$ is a left quasigroup (equivalently, $(X,\circ)$ is a left quasigroup), $x\bullet y = x$ and
\begin{displaymath}
    x\ld\circ(y\ld\circ z) = (x\ld\circ y)\ld\circ (x\ld\circ z).
\end{displaymath}
In terms of the left translations in $(X,\circ)$, the last identity is equivalent to $L_x^{-1}L_y^{-1} = L_{x\ld\circ y}^{-1}L_x^{-1}$. Taking inverses on both sides, we obtain the equivalent identity $L_yL_x = L_x L_{x\ld\circ y}$. Substituting $x\circ y$ for $y$, we obtain $L_{x\circ y}L_x = L_xL_y$, which is \eqref{Eq:DivRack}.

(ii) By Proposition \ref{Pr:OneOp}, $r$ is an involutive left nondegenerate braiding if and only if $(X,\circ)$ is a left quasigroup, $x\bullet y = (x\ld\circ y)\circ x$ and
\begin{displaymath}
    x\ld\circ (y\ld\circ z) = (x\ld\circ y)\ld\circ (((x\ld\circ y)\circ x)\ld\circ z).
\end{displaymath}
This says $L_x^{-1}L_y^{-1} = L_{x\ld\circ y}^{-1}L_{(x\ld\circ y)\circ x}^{-1}$, which is equivalent to $L_{x\circ y}L_x = L_{y\circ x}L_y$, i.e., to \eqref{Eq:DivRump}.

(iii) By Proposition \ref{Pr:OneOp}, $r$ is an idempotent left nondegenerate braiding if and only if $(X,\circ)$ is a left quasigroup, $x\bullet y = (x\ld\circ y)\circ(x\ld\circ y)$ and
\begin{displaymath}
    x\ld\circ (y\ld\circ z) = (x\ld\circ y)\ld\circ (((x\ld\circ y)\circ(x\ld\circ y))\ld\circ z).
\end{displaymath}
This says $L_x^{-1}L_y^{-1} = L_{x\ld\circ y}^{-1}L_{(x\ld\circ y)\circ(x\ld\circ y)}^{-1}$, which is equivalent to $L_{x\circ y}L_x = L_{y\circ y}L_y$, i.e., to \eqref{Eq:DivtWard}.
\end{proof}

Recall from the introduction that a left quasigroup $(X,*)$ is a \emph{rack} (resp. \emph{Rump left quasigroup}, resp. \emph{twisted Ward left quasigroup}) if it satisfies \eqref{Eq:RACK} (resp. \eqref{Eq:RUMP}, resp. \eqref{Eq:tW}).

Part (i) of Theorem \ref{Th:Correspondences} is well-known and part (ii) can be found in \cite[Proposition 1]{Rump}.

\begin{theorem}\label{Th:Correspondences}
Let $X$ be a set. Denote a typical braiding on $X$ by $r(x,y)=(x\circ y,x\bullet y)$. Then:
\begin{enumerate}
\item[(i)] There is a one-to-one correspondence between derived left nondegenerate braidings on $X$ and racks on $X$, given by
\begin{displaymath}
    r\mapsto (X,*),\ x*y = x\ld\circ y,\quad\quad\quad (X,*)\mapsto r,\ r(x,y) = (x\ld* y, x).
\end{displaymath}
\item[(ii)] There is a one-to-one correspondence between involutive left nondegenerate braidings on $X$ and Rump left quasigroups on $X$, given by
\begin{displaymath}
    r\mapsto (X,*),\ x*y = x\ld\circ y,\quad\quad\quad (X,*)\mapsto r,\ r(x,y) = (x\ld* y, (x\ld* y)*x).
\end{displaymath}
\item[(iii)] There is a one-to-one correspondence between idempotent left nondegenerate braidings on $X$ and twisted Ward left quasigroups on $X$, given by
\begin{displaymath}
    r\mapsto (X,*),\ x*y = x\ld\circ y,\quad\quad\quad (X,*)\mapsto r,\ r(x,y) = (x\ld* y, (x\ld* y)*(x\ld* y)).
\end{displaymath}
\end{enumerate}
\end{theorem}
\begin{proof}
It is clear that in each case the two mappings are mutually inverse. The rest follows from Proposition \ref{Pr:DivOneOp}.
\end{proof}

\begin{remark}
(a) As we have shown, the identities $x\circ(y\circ z) = (x\circ y)\circ(x\circ z)$ and $x\ld\circ(y\ld\circ z) = (x\ld\circ y)\ld\circ(x\ld\circ z)$ are equivalent in the variety of left quasigroups. It is therefore customary to replace the correspondence from Theorem \ref{Th:Correspondences}(i) with the correspondence
\begin{displaymath}
    r\mapsto (X,*),\,x*y = x\circ y,\quad\quad\quad (X,*)\mapsto r,\,r(x,y) = (x*y,x).
\end{displaymath}
Our version of the correspondence for racks fits better with the uniform approach employed here.

(b) Neither of the identities \eqref{Eq:Rump} and \eqref{Eq:DivRump} implies the other in the variety of left quasigroups. Likewise, neither of the identities \eqref{Eq:tWard} and \eqref{Eq:DivtWard} implies the other in the variety of left quasigroups. Therefore, in both cases, there are two varieties of left quasigroups that can be chosen to correspond to the braidings in question.
\end{remark}

\begin{corollary}
Let $X$ be a set, $|X|\geq2$. Then every idempotent braiding on $X$ is degenerate.
\end{corollary}
\begin{proof}
Let $r(x,y)=(x\circ y, x\bullet y)$ be an idempotent nondegenerate braiding and let $(X,*)$ be the corresponding twisted Ward left quasigroup, i.e., $x*y=x\ld\circ y$. For every $y\in X$, the right translation $x\mapsto x\bullet y = (x\ld*y)*(x\ld*y)$ is a bijection of $X$, which immediately implies that the squaring mapping $\sigma:x\mapsto x*x$ is onto $X$. We will prove that $\sigma$ is also injective. First observe that $x\bullet(x*y)=(x\ld*(x*y))*(x\ld*(x*y))=y*y$, hence $x=(y*y)\rd\bullet(x*y)$ is independent of $y$, and thus $x=(x*x)\rd\bullet(x*x)$. Consequently, if $\sigma(u)=\sigma(v)$, we have $u=(u*u)\rd\bullet(u*u)=(v*v)\rd\bullet(v*v)=v$.

Now, $\sigma(x*y) = (x*y)*(x*y)=(y*y)*(y*y) = \sigma(y*y)$ is an instance of \eqref{Eq:tW}, and since $\sigma$ is bijective, we have $x*y=\sigma(y)$ for every $x,y\in X$. But then $x\bullet y = (x\ld*y)*(x\ld*y) = \sigma(x\ld*y) = \sigma(\sigma^{-1}(y)) = y$, hence $r$ is right degenerate, a contradiction.
\end{proof}

The following are examples of twisted Ward left quasigroups.

\begin{example}
Let $(X,*)$ be an elementary abelian $2$-group. Then $(x*y)*(x*z) = y*z = (y*y)*(y*z)$, so $(X,*)$ is a twisted Ward (left) quasigroup.
\end{example}

\begin{example}\label{ex:tWq from abelian group}
Let $(X,+)$ be an abelian group, $\varphi\in\mathrm{End}(X,+)$, $\psi\in\aut{X,+}$ and $c\in X$. Define a binary operation $*$ on $X$ by
\begin{displaymath}
    x*y=\varphi(x)+\psi(y)+c.
\end{displaymath}
It is easy to check that the resulting left quasigroup $(X,*)$ satisfies \eqref{Eq:tW} if and only if
\begin{displaymath}
    \varphi\psi=\psi\varphi \quad\text{and}\quad \varphi^2+\varphi\psi=0.
\end{displaymath}
\end{example}

\begin{example}\label{Ex:Perm}
Let $x*y=f(y)$ for some bijection $f$ of $X$. Then $(X,*)$ is clearly a left quasigroup, usually called a \emph{permutational left quasigroup}. Every permutational left quasigroup satisfies \eqref{Eq:tW}.
\end{example}

\begin{lemma}
Let $(X,*)$ be a twisted Ward left quasigroup. Then the following conditions are equivalent:
\begin{enumerate}
\item[(i)] $(X,*)$ is a rack,
\item[(ii)] $(X,*)$ is a Rump left quasigroup,
\item[(iii)] $(X,*)$ is permutational.
\end{enumerate}
\end{lemma}
\begin{proof}
If $(X,*)$ is permutational then it satisfies both \eqref{Eq:RACK} and \eqref{Eq:RUMP}. If $(X,*)$ satisfies \eqref{Eq:RACK} then $x*(y*z) = (x*y)*(x*z) = (y*y)*(y*z)$ and substituting $y\ld*z$ for $z$ yields $x*z = (y*y)*z$, which means that all left translations are the same and $(X,*)$ is permutational. If $(X,*)$ satisfies \eqref{Eq:RUMP} then $(y*x)*(y*z) = (x*y)*(x*z) = (y*y)*(y*z)$ and substituting $y\ld*z$ for $z$ and $y\ld*x$ for $x$ again yields $x*z=(y*y)*z$.
\end{proof}

We will return to twisted Ward left quasigroups in Section \ref{sec:prime}.

\section{Twisted Ward quasigroups}\label{sec:tw-quasigroups}

A quasigroup $(X,*)$ is a \emph{twisted Ward quasigroup} if it satisfies the identity \eqref{Eq:tW}.

Note that in Example \ref{ex:tWq from abelian group}, $(X,*)$ is a quasigroup if and only if $\varphi\in\aut{X,+}$, in which case $(X,*)$ satisfies \eqref{Eq:tW} if and only if $\varphi=-\psi$. This motivates the following construction.

\begin{example}\label{Ex:Main}
Let $(X,\cdot)$ be a group, $\psi\in\aut{X,\cdot}$ and $c\in X$. Then $(X,*)=\mathrm{tWq}(X,\cdot,\psi,c)$ defined by
\begin{displaymath}
    x*y=c\psi(x^{-1}y)
\end{displaymath}
is a twisted Ward quasigroup. Indeed, to verify \eqref{Eq:tW}, we compute
\begin{displaymath}
    (x*y)*(x*z) = (c\psi(x^{-1}y))*(c\psi(x^{-1}z))
    = c\psi((c\psi(x^{-1}y))^{-1}c\psi(x^{-1}z))
    = c\psi(\psi(y^{-1}z)),
\end{displaymath}
which is independent of $x$ and therefore equal to $(y*y)*(y*z)$.
\end{example}

As we shall see in Theorem \ref{Th:tWard}, all twisted Ward quasigroups are of the form $\mathrm{tWq}(X,\cdot,\psi,c)$. We start by proving in two ways that every twisted Ward quasigroup is isotopic to a group. The first proof uses the quadrangle criterion known from the theory of latin squares and the second proof is based on the structure of the displacement group.

Recall that two quasigroups $(X,\cdot)$, $(Y,*)$ are \emph{isotopic} if there are bijections $\alpha,\beta,\gamma:X\to Y$ such that $\alpha(x)*\beta(y)=\gamma(x\cdot y)$ for every $x,y\in X$. Replacing $(Y,*)$ with an isomorphic copy, we can assume that $\gamma=1$ in an isotopism \cite[Theorem III.1.4]{Pfl}.

\begin{proposition}[{{\cite[p. 18]{DK} or \cite[Theorem 2.2]{E}}}]\label{p:QC}
A quasigroup $(X,*)$ is isotopic to a group if and only if it satisfies the \emph{quadrangle criterion}, i.e., for every $a_i,b_i,c_i,d_i\in X$, if $a_1*c_1=a_2*c_2$, $a_1*d_1=a_2*d_2$ and $b_1*c_1=b_2*c_2$ then $b_1*d_1=b_2*d_2$.
\end{proposition}

\begin{lemma}\label{l:dis}
Let $(X,*)$ be a twisted Ward quasigroup and $a_1,a_2,c_1,c_2\in X$. Then:
\begin{enumerate}
	\item[(i)] The squaring map $x\mapsto x*x$ is constant.
	\item[(ii)] If $a_1*c_1=a_2*c_2$, then $L_{a_1}L_{c_1}^{-1} = L_{a_2}L_{c_2}^{-1}$.
\end{enumerate}
\end{lemma}
\begin{proof}
(i) From \eqref{Eq:tW}, $(x*y)*(x*y)=(y*y)*(y*y)$ for every $x,y\in X$. Upon substituting $x\rd*y$ for $x$, we see that $x*x$ is independent of $x$.

(ii) Let us again write $x*y=xy$. Note that $(xx)(x(y\ldiv z)) = (yx)(y(y\ldiv z)) = (yx)z$ is an instance of \eqref{Eq:tW} and therefore $(c_2c_2)(c_2(a_2\ldiv(a_1x))) = (a_2c_2)(a_1x) = (a_1c_1)(a_1x) = (c_1c_1)(c_1x)$, using \eqref{Eq:tW} again in the last step. Canceling the unique square and dividing on the left by $c_2$, we get $a_2\ldiv(a_1x) = c_2\ldiv (c_1x)$. Substituting $c_1\ldiv x$ for $x$ and multiplying on the left by $a_2$, we finally get $a_1(c_1\ldiv x) = a_2(c_2\ldiv x)$.
\end{proof}

\begin{proposition}\label{p:dis}
Every twisted Ward quasigroup is isotopic to a group.
\end{proposition}

\begin{proof}
Suppose that $(X,\cdot)$ is a twisted Ward quasigroup and the assumptions of the quadrangle criterion are satisfied.
Applying Lemma \ref{l:dis}(ii) to equalities $a_1c_1=a_2c_2$ and $b_1c_1=b_2c_2$, we get $d_2 = a_2\ldiv(a_2d_2) = a_2\ldiv(a_1d_1) = a_2\ldiv(a_1(c_1\ldiv(c_1d_1))) = a_2\ldiv (a_2(c_2\ldiv(c_1d_1))) = c_2\ldiv(c_1d_1)$ and thus $b_2d_2 = b_2(c_2\ldiv(c_1d_1)) = b_1(c_1\ldiv (c_1d_1)) = b_1d_1$. We are done by Proposition \ref{p:QC}.
\end{proof}

The \emph{left multiplication group} of a quasigroup $(X,*)$ is the permutation group generated by all left translations, i.e.,
\begin{displaymath}
    \lmlt X=\langle L_x: x\in X\rangle.
\end{displaymath}
As in \cite{BKSV}, we define the \emph{positive} (resp. \emph{negative}) \emph{displacement group} of $(X,*)$ as the subgroup of $\lmlt X$ generated by all positive (resp. negative) displacements, that is,
\begin{displaymath}
    \disp X=\langle L_xL_y^{-1} : x,y\in X\rangle, \quad\quad\quad \disn X=\langle L_x^{-1}L_y : x,y\in X\rangle.
\end{displaymath}
The \emph{displacement group} of $(X,*)$ is then the group
\begin{displaymath}
    \dis X=\langle L_xL_y^{-1}, L_x^{-1}L_y : x,y\in X\rangle.
\end{displaymath}
Note that $\disp X=\langle L_eL_x^{-1} : x\in X\rangle$ for any fixed $e\in X$ since $L_xL_y^{-1}=(L_eL_x^{-1})^{-1}(L_eL_y^{-1})$, and $\disn X=\langle L_x^{-1}L_e :\ x\in X\rangle$  since $L_x^{-1}L_y = (L_x^{-1}L_e)(L_y^{-1}L_e)^{-1}$.

A permutation group $G$ acts \emph{regularly} on a set $X$ if for every $x,y\in X$ there is a unique $g\in G$ such that $g(x)=y$. Recall the following result of Dr\'apal:

\begin{proposition}\cite[Proposition 5.2]{Dra}\label{Pr:isotope}
A quasigroup $X$ is isotopic to a group if and only if $\disp{X}$ acts regularly on $X$. In such a case, $X$ is isotopic to $\disp{X}$.
\end{proposition}

It follows from Propositions \ref{p:dis} and \ref{Pr:isotope} that every twisted Ward quasigroup $X$ is isotopic to the group $\disp{X}$, which acts regularly on $X$. Here is an alternative proof of this fact which does not refer to Propositions \ref{p:QC} and \ref{Pr:isotope}.

\begin{lemma}
Let $X=(X,*)$ be a twisted Ward quasigroup and let $e$ denote the unique square in $X$. Then:
\begin{enumerate}
\item[(i)] $\dis{X}=\disp{X}=\disn{X}$.
\item[(ii)] $(L_x^{-1}L_e)(L_y^{-1}L_e) = L_{(x\rd*e)*(e\ld*y)}^{-1}L_e$ and $(L_x^{-1}L_e)^{-1} = L_{(e\ld*x)*e}^{-1}L_e$ for every $x,y\in X$.
\item[(iii)] $\dis{X}$ is equal to $\{L_x^{-1}L_e:x\in X\}$ and is isomorphic to the group isotope $(X,\diamond)$, where $x\diamond y = (x\rd*e)*(e\ld*y)$.
\end{enumerate}
\end{lemma}

\begin{proof}
(i) The identity \eqref{Eq:tW} says $L_{x*y}L_x = L_{y*y}L_y = L_eL_y$ and hence is equivalent to $L_xL_y^{-1}=L_{x*y}^{-1}L_e$, which shows that $\disp X\le \disn X$. Replacing $y$ with $x\ld*y$, we obtain the identity
\begin{equation}\label{Eq:dis}
    L_xL_{x\ld*y}^{-1}=L_y^{-1}L_e,
\end{equation}
which implies $\disn X\le \disp X$. Hence $\dis X = \disp X = \disn X$.

(ii) Fix $x,y\in X$ and let $u=x\rd*e$ so that $u\ld*x = e$. By a repeated application of \eqref{Eq:dis}, we have
\begin{align*}
    (L_x^{-1}L_e)(L_y^{-1}L_e) &= (L_uL_{u\ld*x}^{-1})(L_eL_{e\ld*y}^{-1}) = L_u(L_{u\ld*x}^{-1}L_e)L_{e\ld*y}^{-1}\\
    &= L_uL_{e\ld*y}^{-1} = L_uL_{u\ld*(u*(e\ld*y))}^{-1} = L_{u*(e\ld*y)}^{-1}L_e = L_{(x\rd*e)*(e\ld*y)}^{-1}L_e.
\end{align*}
Using \eqref{Eq:dis} again, we also have
\begin{displaymath}
    (L_x^{-1}L_e)^{-1} = (L_eL_{e\ld*x}^{-1})^{-1} = L_{e\ld*x}L_e^{-1} = L_{e\ld*x}L_{(e\ld*x)\ld*((e\ld*x)*e)}^{-1} = L_{(e\ld*x)*e}^{-1}L_e.
\end{displaymath}

(iii) Part (ii) proves that $\dis X =\disn X$ is equal to $\{L_x^{-1}L_e:x\in X\}$ and is isomorphic to $(X,\diamond)$, where $x\diamond y = (x\rd*e)*(e\ld*y)$, which therefore has to be a group. Clearly, $(X,\diamond)$ is isotopic to $(X,*)$.
\end{proof}

We proceed to describe all twisted Ward quasigroups.

\begin{theorem}\label{Th:tWard}
Let $(X,*)$ be a quasigroup. Then $(X,*)$ is a twisted Ward quasigroup if and only if there is a group $(X,\cdot)$, $\psi\in\aut{X,\cdot}$ and $c\in X$ such that
$x*y=c\psi(x^{-1}y)$ for every $x,y\in X$.
\end{theorem}

\begin{proof}
We have verified the converse implication in Example \ref{Ex:Main}. For the direct implication, suppose that $(X,*)$ is a twisted Ward quasigroup. By Proposition \ref{p:dis}, $(X,*)$ is isotopic to a group $(X,\cdot)$, i.e., there are permutations $\varphi$, $\psi$ of $X$  such that $x*y=\varphi(x)\psi(y)$ for all $x,y\in X$. We may assume without loss of generality that $\psi(1) = 1$, otherwise set $\bar{\varphi}(x) = \varphi(x)\psi(1)$ and $\bar{\psi}(y) = \psi(1)^{-1}\psi(y)$ to obtain $x*y = \bar{\varphi}(x)\bar{\psi}(y)$ and $\bar\psi(1)=1$.

Writing \eqref{Eq:tW} in terms of $\cdot$, $\varphi$ and $\psi$, and replacing $z$ with $\psi^{-1}(z)$, we have
\begin{displaymath}
    \varphi(\varphi(x)\psi(y))\cdot \psi(\varphi(x) z) = \varphi(\varphi(y)\psi(y))\cdot \psi(\varphi(y) z)
\end{displaymath}
for all $x,y,z\in X$. Rearranging this, we have
\begin{displaymath}
    \varphi(\varphi(y) \psi(y))^{-1}\cdot \varphi(\varphi(x) \psi(y)) = \psi(\varphi(y) z)\cdot \psi(\varphi(x) z)^{-1}.
\end{displaymath}
Note that the left hand side is independent of $z$. Comparing the right hand sides upon substituting $z=1$ and $z=\varphi(x)^{-1}$, respectively, we obtain
\begin{displaymath}
    \psi(\varphi(y)) \cdot \psi(\varphi(x))^{-1} = \psi(\varphi(y) \varphi(x)^{-1})\cdot\psi(\varphi(x)\varphi(x)^{-1}) =\psi(\varphi(y)\varphi(x)^{-1})
\end{displaymath}
for all $x,y\in X$, where we have used $\psi(1)=1$. Since $\varphi$ is bijective, this is equivalent to
\begin{displaymath}
    \psi(y)\psi(x)^{-1} = \psi(yx^{-1})
\end{displaymath}
for all $x,y\in X$, and thus $\psi$ is an automorphism of the group $(X,\cdot)$.
	
Writing \eqref{Eq:tW} in terms of $\cdot$, $\varphi$ and $\psi$ again and substituting $y=z=1$, we have
\begin{displaymath}
    \varphi(\varphi(x))\cdot \psi(\varphi(x)) = \varphi(\varphi(1))\cdot \psi(\varphi(1))
\end{displaymath}
for every $x\in X$, with the right hand side being constant, say equal to $c$. Since $\varphi$ is bijective, this is equivalent to $\varphi(x)\psi(x)=c$. Then $\varphi(x)=c\psi(x)^{-1}$ and $x*y=c\psi(x)^{-1}\psi(y)=c\psi(x^{-1}y)$ for every $x$, $y\in X$.	
\end{proof}

Since isotopic groups are isomorphic \cite[Corollary III.2.3]{Pfl}, the following result solves the isomorphism problem for twisted Ward quasigroups.

\begin{proposition}\label{Pr:tWard_iso}
Let $(X,\cdot)$ be a group, $\varphi,\psi\in\aut{X,\cdot}$ and $c\in X$. Then:
\begin{enumerate}
\item[(i)] The mapping $x\mapsto cx$ is an isomorphism $\mathrm{tWq}(X,\cdot,\varphi,1)\to\mathrm{tWq}(X,\cdot,\varphi,c)$.
\item[(ii)] The twisted Ward quasigroups $\mathrm{tWq}(X,\cdot,\varphi,1)$ and $\mathrm{tWq}(X,\cdot,\psi,1)$ are isomorphic if and only if $\varphi$, $\psi$ are conjugate in $\aut{X,\cdot}$.
\end{enumerate}
\end{proposition}

\begin{proof}
Let us denote the multiplication in $\mathrm{tWq}(X,\cdot,\varphi,c)$ by $*_{\varphi,c}$.

(i) For every $x,y\in X$, we have $(cx)*_{\varphi,c}(cy)=c\varphi((cx)^{-1}(cy))=c\varphi(x^{-1}y)=c(x*_{\varphi,1}y)$.

(ii) If $\psi=\rho\varphi\rho^{-1}$ for some $\rho\in\aut{X,\cdot}$, then
\begin{displaymath}
    \rho(x*_{\varphi,1}y)=\rho\varphi(x^{-1}y)=\psi\rho(x^{-1}y)=\psi(\rho(x)^{-1}\rho(y))=\rho(x)*_{\psi,1}\rho(y).
\end{displaymath}
Conversely, if $\rho$ is an isomorphism $\mathrm{tWq}(X,\cdot,\varphi,1)\to \mathrm{tWq}(X,\cdot,\psi,1)$ then
\begin{equation}\label{Eq:AuxIso}
    \rho\varphi(x^{-1}y)=\psi(\rho(x)^{-1}\rho(y))
\end{equation}
for every $x,y\in X$. Upon substituting $x=1$ into \eqref{Eq:AuxIso} we obtain $\rho\varphi=\psi\rho$. Applying $\psi^{-1}$ to both sides of \eqref{Eq:AuxIso}, we then get $\rho(x^{-1}y)=\rho(x)^{-1}\rho(y)$, so $\rho\in\aut{X,\cdot}$.
\end{proof}

For a group $G$, let $cc(G)$ denote the number of conjugacy classes of $G$.

\begin{corollary}\label{Cr:NtWq}
The number of twisted Ward quasigroups of order $n$ up to isomorphism is
\begin{displaymath}
    q(n) = \sum_{G} cc(\aut{G}),
\end{displaymath}
where the summation runs over all groups $G$ of order $n$ up to isomorphism.
\end{corollary}

Returning to braidings, we deduce:

\begin{theorem}\label{Th:main}
Let $r:X\times X\to X\times X$ be given by $r(x,y)=(x\circ y,x\bullet y)$. Then $r$ is an idempotent latin braiding if and only if there is a group $(X,\cdot)$, an automorphism $\varphi\in\aut{X,\cdot}$ and $c\in X$ such that
\begin{displaymath}
    r(x,y)=(x\varphi(c)^{-1}\varphi(y),c).
\end{displaymath}
Moreover, up to isomorphism, we can take $c=1$ and $\varphi$ up to conjugation in $\aut{X,\cdot}$.
\end{theorem}
\begin{proof}
By Theorem \ref{Th:Correspondences}, $r$ is an idempotent left nondegenerate braiding if and only if $r(x,y) = (x\ld*y,(x\ld*y)*(x\ld*y))$ for a twisted Ward left quasigroup $(X,*)$, and $r$ is latin if and only if $(X,*)$ is a twisted Ward quasigroup. Then, by Theorem \ref{Th:tWard}, $x*y = c\psi(x^{-1}y)$ for some group $(X,\cdot)$, $\psi\in\aut{X,\cdot}$ and $c\in X$. Then certainly $(x\ld*y)*(x\ld*y) = c$, and since $x\ld*y = x\psi^{-1}(c^{-1}y)$, we can write $x\circ y = x\ld*y = x\varphi(c)^{-1}\varphi(y)$ by taking $\varphi=\psi^{-1}$. The last part follows by Proposition \ref{Pr:tWard_iso}.
\end{proof}

\section{The Cayley kernel, squaring and twisted Ward left quasigroups of prime order}\label{sec:prime}

For an equivalence relation $R$ on a set $X$, let $X_R$ be a complete set of representatives of the equivalence classes of $R$ and let $[x]_R$ denote the equivalence class of $R$ containing the element $x$.

On a left quasigroup $(X,\cdot)$ define two equivalence relations
\begin{align*}
    &x\sim y\quad\Leftrightarrow\quad L_x = L_y,\\
    &x\equiv y\quad\Leftrightarrow\quad xx=yy.
\end{align*}
The equivalence relation $\sim$ is usually called the \emph{Cayley kernel} in this context. If $\sim$ is the full equivalence $X\times X$ then $(X,\cdot)$ is called \emph{permutational}, cf. Example \ref{Ex:Perm}. If $\sim$ is the equality relation $\{(x,x):x\in X\}$ then $(X,\cdot)$ is called \emph{faithful}.

\begin{lemma}\label{Lm:TwoEquivalences}
Let $(X,\cdot)$ be a left quasigroup. Then the equivalence relations $\sim$ and $\equiv$ intersect trivially and the following inequalities hold for every $x\in X$:
\begin{enumerate}
\item[(i)] $|[x]_\equiv|\le |X_\sim|$,
\item[(ii)] $|[x]_\sim|\le |X_\equiv|$,
\item[(iii)] $|X|\le |X_\sim|\cdot |X_\equiv|$.
\end{enumerate}
\end{lemma}
\begin{proof}
Let $x\sim y$ and $x\equiv y$. Then $xy=yy=xx$ and left cancellation yields $x=y$. The inequalities (i) and (ii) are then immediate consequences. Finally, $|X| = \sum_{x\in X_\sim} |[x]_\sim| \le |X_\sim|\cdot |X_\equiv|$ by (ii).
\end{proof}

In general left quasigroups, neither of the two equivalences is a congruence. In finite twisted Ward left quasigroups, the Cayley kernel is not always a congruence (cf. Example \ref{Ex:NotCong}) but $\equiv$ is a congruence (cf. Proposition \ref{Pr:equiv is con}). We do not know whether $\equiv$ is a congruence in infinite Ward left quasigroups, too.

\begin{example}\label{Ex:NotCong}
In the twisted Ward left quasigroup with multiplication table
\begin{displaymath}
\begin{array}{r|rrrr}
    &1&2&3&4\\
    \hline
    1&1&3&2&4\\
    2&1&3&2&4\\
    3&4&2&3&1\\
    4&4&2&3&1
\end{array}
\end{displaymath}
the Cayley kernel is not a congruence: $L_1=L_2$ and $L_{2\cdot 1} = L_1\ne L_3 = L_{2\cdot 2}$.
\end{example}

\begin{proposition}\label{Pr:equiv is con}
In a finite twisted Ward left quasigroup, the equivalence $\equiv$ is a congruence.
\end{proposition}

\begin{proof}
Since $(xz)(xz)$ does not depend on $x$ by \eqref{Eq:tW}, we always have $xz\equiv yz$. If $x\equiv y$ then  $(zx)(zx) = (xx)(xx)=(yy)(yy) = (zy)(zy)$ by \eqref{Eq:tW}, and hence $zx\equiv zy$. By finiteness, $\equiv$ is invariant under left division, too.
\end{proof}

We proceed towards a classification of twisted Ward left quasigroups of prime order.

\begin{lemma}\label{Lm:SameRow}
Let $(X,\cdot)$ be a twisted Ward left quasigroup and $x,y\in X$. If $L_x$, $L_y$ agree at a point then $L_x=L_y$. In particular, every finite faithful twisted Ward left quasigroup is a quasigroup.
\end{lemma}
\begin{proof}
Assume that $xc=yc$ for some $c\in X$. Then for every $z\in X$ we have $(xc)(xz) = (cc)(cz)= (yc)(yz)=(xc)(yz)$ and we obtain $xz=yz$ by left cancellation.
\end{proof}

\begin{proposition}\label{p:TwoEquivalences}
Let $(X,\cdot)$ be a twisted Ward left quasigroup. Then the following conditions hold for every $x\in X$:
\begin{enumerate}
\item[(i)] $|[x]_\equiv|=|X_\sim|$,
\item[(ii)] if $X$ is finite then $|[x]_\sim|=|X_\equiv|$,
\item[(iii)] $|X|=|X_\equiv|\cdot |X_\sim|$.
\end{enumerate}
\end{proposition}
\begin{proof}
(i) Fix $z\in X$ and observe that for every $x\in X$ we have $(xz)(xz) = (zz)(zz)$. Hence the elements $R_z(X)=\{xz:x\in X\}$ in the column indexed by $z$ all have the same square. By Lemma \ref{Lm:SameRow}, the cardinality of $R_z(X)$ is equal to $|X_\sim|$. Hence $|[u]_\equiv|\geq|X_\sim|$ for every $u\in R_z(x)$, and Lemma \ref{Lm:TwoEquivalences}(i) asserts equality. By varying $z$, we will encounter all elements of $X$ in this fashion.

(iii) By (i), all blocks of $\equiv$ have the same size $|X_\sim|$, so $|X| = |X_\equiv|\cdot |X_\sim|$.

(ii) By Lemma \ref{Lm:TwoEquivalences}(ii) and by part (iii), we have $|X|=\sum_{x\in X_\sim} |[x]_\sim| \leq |X_\sim|\cdot |X_\equiv| = |X|$. Hence we have an equality and $|[x]_\sim|=|X_\equiv|$ follows because $|[x]_\sim|\le |X_\equiv|$ for every $x\in X$.
\end{proof}

\begin{example}
The sets $X_\sim$, $X_\equiv$ may have different cardinalities. For instance, in the twisted Ward left quasigroup $X$ with multiplication table
\begin{displaymath}
\begin{array}{r|rrrrrr}
    &1&2&3&4&5&6\\
    \hline
    1&2&1&4&3&5&6\\
    2&3&4&1&2&6&5\\
    3&2&1&4&3&5&6\\
    4&3&4&1&2&6&5\\
    5&2&1&4&3&5&6\\
    6&3&4&1&2&6&5
\end{array}
\end{displaymath}
we have $|[x]_\sim|=3$ and $|[x]_\equiv|=2$ for every $x\in X$. Note that $X$ is neither permutational nor a quasigroup.
\end{example}

\begin{theorem}\label{Th:Prime}
Every twisted Ward left quasigroup of prime order is either permutational or a quasigroup.
\end{theorem}
\begin{proof}
Let $X$ be a twisted Ward left quasigroup of prime order. By Proposition \ref{p:TwoEquivalences}, all equivalence classes of $\sim$ have the same cardinality. Since $X$ is of prime order, it follows that either $\sim$ has a single equivalence class or all equivalence classes of $\sim$ are singletons. In the former case, $X$ is permutational. In the latter case, $X$ is faithful and thus a quasigroup by Lemma \ref{Lm:SameRow}.
\end{proof}

\begin{corollary}\label{Cr:NtWlq}
Let $q(n)$ (resp. $\ell(n)$) denote the number of twisted Ward quasigroups (resp. twisted Ward left quasigroups) of order $n$ up to isomorphism. Let $p(n)$ be the partition number, i.e., the number of ways in which $n$ can be written as a sum of nonincreasing positive integers. Then $\ell(n)\ge q(n)+p(n)$ if $n>1$. Moreover, if $n$ is prime then $\ell(n) = q(n)+p(n) = n-1+p(n)$.
\end{corollary}
\begin{proof}
Denote by $(X,f)$ the permutational (Ward) left quasigroup with multiplication given by $x*y=f(y)$. It is easy to see that $(X,f)$ is isomorphic to $(X,g)$ if and only if $f$ and $g$ are conjugate in the symmetric group $S_X$. Recall that there are $p(|X|)$ conjugacy classes in $S_X$. Moreover, if $|X|>1$ then $(X,f)$ is never a quasigroup. Hence $\ell(n)\ge q(n)+p(n)$ if $n>1$.

Suppose that $n$ is prime. By Theorem \ref{Th:Prime}, $\ell(n) = q(n)+p(n)$. The only group of order $n$ is the cyclic group $C_n$. Since $\aut{C_n}$ is an abelian group of order $n-1$, we have $q(n) = n-1$ by Corollary \ref{Cr:NtWq}.
\end{proof}

The following table summarizes the numbers of twisted Ward left quasigroups $\ell(n)$ and twisted Ward quasigroups $q(n)$ of order $n\le 11$ up to isomorphism, as well as the partition number $p(n)$.
\begin{displaymath}
    \begin{array}{rrrrrrrrrrrr}
        n&1&2&3&4&5&6&7&8&9&10&11\\
        \hline
        \ell(n)&1&3&5&14&11&31&21&93&64&?&66\\
        q(n)&1&1&2&5&4&5&6&25&14&9&10\\
        p(n)&1&2&3&5&7&11&15&22&30&42&56\\
    \end{array}
\end{displaymath}
Neither of the sequences $(\ell(n))$, $(q(n))$ appears in the Online Encyclopedia of Integer Sequences \cite{OEIS}. It is not difficult to calculate the numbers $q(n)$ for small values of $n$ by hand, using Corollary \ref{Cr:NtWq}. We can then calculate $\ell(n)$ for prime orders $n$ from Corollary \ref{Cr:NtWlq}. The remaining values $\ell(n)$ (and for independent verification also all other values except for $\ell(10)$ and $\ell(11)$) were calculated by the finite model builder \texttt{Mace4} \cite{McCune}.

We conclude the paper with a construction that yields all finite twisted left Ward quasigroups in principle.

\begin{proposition}
Let $X$ and $A$ be sets. For every $x\in X$, let $f_x$ be a bijection on $X\times A$ and let us write $f_x(y,b) = (f_x^{[1]}(y,b),f_x^{[2]}(y,b))$. Define $(X\times A,*)$ by
\begin{displaymath}
    (x,a)*(y,b) = f_x(y,b).
\end{displaymath}
Then $(X\times A,*)$ is a twisted Ward left quasigroup if and only if
\begin{equation}\label{Eq:Construction}
    f_{f_x^{[1]}(y,b)}f_x
\end{equation}
is independent of $x$. Moreover, every finite twisted Ward left quasigroup is isomorphic to one of this form.
\end{proposition}
\begin{proof}
The groupoid $(X\times A,*)$ is a left quasigroup. The identity \eqref{Eq:tW} requires that the expression $((x,a)*(y,b))*((x,a)*(z,c))$ is idependent of $(x,a)$. Expanding the expression, we obtain
\[
     f_x(y,b)*f_x(z,c) = (f_x^{[1]}(y,b),f_x^{[2]}(y,b))*f_x(z,c) = f_{f_x^{[1]}(y,b)}f_x(z,c).
\]
Hence $(X\times A,*)$ satisfies \eqref{Eq:tW} if and only if \eqref{Eq:Construction} is independent of $x$.

By Proposition \ref{p:TwoEquivalences}, the Cayley kernel of a finite twisted Ward left quasigroup $W$ has blocks of the same size, say each bijectively mapped onto a fixed set $A$. We can then represent the underlying set of $W$ as $X\times A$ for a suitable set $X$. The product $(x,a)$ and $(y,b)$ in $W$ depends only on $x$, $y$ and $b$. Moreover, for a fixed $(x,a)$, the left translation by $(x,a)$ in $W$ is a bijection of $X\times A$ depending on $x$ only, and this is how we obtain the mappings $f_x$.
\end{proof}

\end{document}